\documentclass[12pt]{article}

\usepackage{amsmath}
\usepackage{mathtools}
\usepackage{amsfonts}
\usepackage{amssymb}
\usepackage{graphicx}
\usepackage{hyperref}
\usepackage{amsthm}

\newtheorem{theorem}{Theorem}
\newtheorem{lemma}[theorem]{Lemma}
\newtheorem{conjecture}[theorem]{Conjecture}
\newtheorem{proposition}[theorem]{Proposition}

\newcommand{\e}{\expectation}

\newcommand{\giventhat}{\mid}
\newcommand{\evaluatedat}[2]{\left.{#1}\right\rvert_{#2}}
\newcommand{\RR}{\mathbb{R}}
\newcommand{\expectation}{\operatorname{\mathbb{E}}}
\newcommand{\vol}{\operatorname{vol}}
\newcommand{\conv}{\operatorname{conv}}
\newcommand{\suchthat}{\mathrel{:}}
\newcommand{\inner}[2]{#1\cdot #2}
\newcommand{\norms}[1]{{\lVert#1\rVert}^2}
\newcommand{\bdry}{\operatorname{bdry}}
\newcommand{\norm}[1]{{\lVert#1\rVert}}
\newcommand{\eps}{\epsilon}
\newcommand{\dbar}[1]{\bar{\bar{#1}}}
\newcommand{\abs}[1]{\left\lvert#1\right\rvert}
\newcommand{\rowmatrix}[4]{
\begin{pmatrix}
#1 & #2 \\
#3 & #4
\end{pmatrix}
}
\newcommand{\keywords}[1]{\emph{Keywords:} #1}

\newcommand{\email}[1]{\href{mailto:#1}{\texttt{#1}}}

\title{On the monotonicity of the expected volume of a random simplex}

\author{Luis Rademacher\footnote{Part of this work was done while the author was a Postdoctoral Fellow at the College of Computing of the Georgia Institute of Technology.}\\
Computer Science and Engineering\\
Ohio State University\\
\email{lrademac@cse.ohio-state.edu}}

\date{}

\begin{document}
\maketitle
\begin{abstract}
Let a random simplex in a $d$-dimensional convex body be the convex hull of $d+1$ random points from the body. We study the following question: As a function of the convex body, is the expected volume of a random simplex monotone non-decreasing under inclusion? We show that this holds if $d$ is 1 or 2, and does not hold if $d \geq 4$. We also prove similar results for higher moments of the volume of a random simplex, in particular for the second moment, which corresponds to  the determinant of the covariance matrix of the convex body. These questions are motivated by the slicing conjecture.
\end{abstract}


\section{Introduction}
For a $d$-dimensional convex body $K$, let $V_K$ denote the (random) volume of the convex hull of $d+1$ independent random points in $K$. In \cite{MeckesSanAntonio}, Mark Meckes asked whether for any pair of convex bodies $K,L \subseteq \RR^d$, $K \subseteq L$ implies
\[
\e (V_K) \leq \e (V_L).
 \]
His ``strong conjecture'' claims that this holds. He also stated the following ``weak conjecture'': there exists a universal constant $c>0$ such that $K \subseteq L$ implies
\[
\e (V_K) \leq c^d \e (V_L).
\]
Clearly, the strong conjecture implies the weak conjecture.
He also asked about the natural generalization to more than $d+1$ points, random polytopes, and higher moments.
Matthias Reitzner also discusses the problem later in \cite{ReitznerToulouse}; he asked whether $K \subseteq L$ implies
\[
\e_{X_0, \dotsc, X_n \in K} (\vol \conv X_0,\dotsc, X_n) \leq \e_{X_0, \dotsc, X_n \in L} ( \vol \conv X_0,\dotsc, X_n)
\]
for arbitrary $n$.

While these are natural questions in the understanding of random polytopes, one of their main motivations comes from their connection with the slicing conjecture (also known as the hyperplane conjecture or slicing problem): All $d$-dimensional convex bodies of volume 1 have a hyperplane section of $(d-1)$-dimensional volume at least a universal positive constant. Meckes's weak conjecture is equivalent to the slicing conjecture (see the Appendix). The slicing conjecture \cite{Bourgain1986, Ball1988, Giannopoulos} is one of the outstanding open problems in asymptotic convex geometry; one of the reasons is its connections with classical problems in convexity, like the Busemann-Petty problem and Sylvester's problem \cite{Giannopoulos}.

In this work we show that Meckes's strong conjecture has a negative answer if $d\geq 4$ and a positive answer if $d$ is $1$ or $2$. More precisely, we show:
\begin{theorem}[random simplex]\label{thm:simplexMonotonicity}
If $d$ is $1$ or $2$, and $K$, $L$ are two $d$-dimensional convex bodies, then $K\subseteq L$ implies
\[
\e (V_K) \leq \e (V_L).
 \]
If $d \geq 4$, then there exist two convex bodies $K \subseteq L \subseteq \RR^d$ such that
\[
\e (V_K) > \e (V_L).
\]
\end{theorem}
For the case $d=3$, numerical integration suggests that the same counterexample used for $d \geq 4$ works for $d=3$. Certain approximations used in those integrals in the proof for higher $d$ fail to give a proof for $d=3$, while an exact evaluation of the integrals looks somewhat involved and is left as an open question.

From the proof of Theorem \ref{thm:simplexMonotonicity} one can infer the following counterexample: In $d$ dimensions, let $L$ be the convex hull of a half-ball (say, the unit ball with the constraint $x_1\geq 0$) and a point at distance $\epsilon>0$ from the center of the ball (say, the point $(-\epsilon, 0, \dotsc, 0)$). That is, $L$ is the union of a half-ball and a cone. Let $K$ be $L$ with the tip of the cone truncated at distance $\delta>0$ (say, $K = L \cap \{x \suchthat x_1 \geq -\epsilon + \delta\}$). Then the proof of Theorem \ref{thm:simplexMonotonicity} shows that the pair $K$, $L$ is a counterexample to the monotonicity for $d \geq 4$ and $\epsilon$, $\delta$ sufficiently small. Numerical integration suggests the same for $d=3$.

The same counterexample and analysis work for higher moments and all dimensions larger than 1 (it is easy to see that for dimension 1 the monotonicity holds for all moments). More precisely, we show:
\begin{theorem}[higher moments]\label{thm:higherMoments}
If $d$ is 2 or 3, then there exist an integer $k_0 \geq 1$ and two convex bodies $K \subseteq L \subseteq \RR^d$ such that for any integer $k \geq k_0$ we have
\[
\e (V_K^k ) > \e (V_L^k ).
\]
If $d \geq 4$, then there exist two convex bodies $K \subseteq L \subseteq \RR^d$ such that for any integer $k \geq 1$ we have
\[
\e (V_K^k ) > \e (V_L^k ).
\]
\end{theorem}

The intuition for our answer to Meckes's question came from our solution to another simpler but related question asked by Santosh Vempala: is the determinant of the covariance matrix of a convex body monotone under inclusion? (The covariance matrix $A(\cdot)$ is defined in Section \ref{sec:preliminaries}) Here we show:
\begin{theorem}[determinant of covariance]\label{thm:detcov}
If $d$ is $1$ or $2$ and $K, L$ are two $d$-dimensional convex bodies, then $K\subseteq L$ implies $\det A(K) \leq \det A(L)$. If $d \geq 3$, then there exist two convex bodies $K \subseteq L \subseteq \RR^d$ such that $ \det A(K) > \det A(L)$.
\end{theorem}

The question by Vempala was also motivated by the slicing conjecture. As we show in the Appendix, the following weaker version of it is equivalent to the slicing conjecture: There exists a universal constant $c_1>0$ such that for any pair of convex bodies $K, M \subseteq \RR^d$ we have
\[
K \subseteq M \implies \det A(K) \leq c_1^{d} \det A(M).
\]

The high level idea of the proof of Theorem \ref{thm:detcov} is the following: To understand the monotonicity it is enough to compute and understand the derivative of $\det A(\cdot)$ as one intersects the convex body with a moving halfspace (Proposition \ref{pro:derivativeAK}). We then find conditions under which this derivative has always the right sign (Lemma \ref{lem:detCharacterization} and the proof of Theorem \ref{thm:detcov}). In the proof of Theorem \ref{thm:detcov} it is shown that understanding such a derivative is enough.

The following formula explains the connection between the determinant of the covariance matrix and the expected volume of a random simplex:
\begin{lemma}\label{lem:detVSsimplex}
Let $K$ be a $d$-dimensional convex body. Let $X_1, \dotsc, X_d$ be random in $K$. Let $\mu(K):=\e X_1$ be the centroid of $K$. Then
\begin{align}
\det A(K) &= d! \e_{X_i \in K} \bigl((\vol \conv \mu(K), X_1, \dotsc, X_d)^2\bigr) \nonumber \\
    &= \frac{d!}{d+1} \e (V_K^2),\label{equ:detVSsimplex}
\end{align}
\end{lemma}
The first equality is known and easy to verify, see e.g. \cite[Proposition 1.3.3]{Giannopoulos}, the second equality is a slight extension, see Section \ref{sec:proofs} for a proof.

In view of Equation \ref{equ:detVSsimplex}, one would think that if a pair of convex bodies is an example that the monotonicity of $\det A(\cdot)$ does not hold, then it is also such an example for the functional
\[
K \mapsto \e (V_K).
\]
Given these similarities, it should be no surprise that techniques and examples similar to those for $\det A(\cdot)$ also work for the expected volume of a random simplex and higher moments.

For the proof of Theorem \ref{thm:simplexMonotonicity} we use a special case of Crofton's theorem\footnote{Sometimes called Crofton's differential equation. It gives an expression for the derivative of a symmetric function of random points from a domain as the domain is perturbed.} \cite[Chapter 5]{Solomon}, \cite[Chapter 2]{KendallMoran}. Our special case is Proposition \ref{pro:Crofton}, which we prove here for completeness, partly because the proof of this version is elementary and because originally the formula was an informal statement instead of a theorem. Crofton's theorem has been formalized at least twice, once with differential geometry \cite{Badderley} and another time with conditional probability \cite{EisenbergSullivan}. It is likely that using either of these two versions one could prove Theorem \ref{thm:simplexMonotonicity} in a simpler but less elementary way.

\section{Preliminaries}\label{sec:preliminaries}

Let $K \subseteq \RR^d$ be a convex body. Let $X_0, \dotsc, X_d$ be random points in $K$.
Let $V_K$ denote the random variable $V_K = \vol (\conv (X_0, \dotsc, X_d))$.
Let $\vol(\cdot)$ be the $d$-dimensional volume function. Let $\vol_k(\cdot)$ be the $k$-dimensional volume function.

Let $X$ be random in $K$. Let $\mu(K)$ denote the centroid of $K$: $\mu(K) = \e (X)$. Let $A(K)$ be the covariance matrix of $K$:
\[
A(K) = \e \Bigl(\bigl(X-\mu(K)\bigr) \bigl(X-\mu(K)\bigr)^T\Bigr).
\]
We say that $K$ is isotropic iff $\mu(K) = 0$ and $A(K)$ is the identity matrix. It is easy to see that any convex body can be made isotropic by applying an affine transformation to it.

Given $K$ and a hyperplane $H$, the \emph{Steiner symmetrization} of $K$ with respect to $H$ is the convex body that results from the following process: For every line $L$ orthogonal to $H$ such that the segment $L \cap K$ is non-empty, shift the segment along $L$ so that its midpoint lies in $H$. Similarly, given $K$ and a halfspace $H$, \emph{Blaschke's shaking (Sch\"uttelung)} of $K$ with respect to $H$ is the convex body that results from the following process: For every line $L$ orthogonal to $H$ such that the segment $L \cap K$ is non-empty, shift the segment along $L$ so that one endpoint lies on the boundary of $H$ while the whole segment stays inside $H$ (See \cite{Pfiefer} for a discussion). If a given hyperplane $H$ does not intersect the interior of $K$, then we define Blaschke's shaking of $K$ with respect to $H$ as the shaking defined before with respect to the halfspace containing $K$ and having $H$ as boundary.

Whenever we have a function defined on an interval $[a,b]$ and we write the derivative of $f$ at $a$ we (implicitly) mean the one-sided derivative.

We will need the following known results:

\begin{theorem}[Blaschke, {\protect \cite[Note 1 for Section 8.2.3]{schneider}}]\label{thm:blaschke}
For any $2$-dimensional convex body $K$:
\[
\frac{\e( V_K)}{\vol(K)} \leq \frac{1}{12},
\]
with equality iff $K$ is a triangle.
\end{theorem}

\begin{theorem}[Blaschke-Groemer, {\protect \cite[Theorem 8.6.3]{schneider}}]\label{thm:blaschkegroemer}
Let $k \geq 1$ be an integer. Among all $d$-dimensional convex bodies,
\[
K \mapsto \frac{\e( V_K^k)}{\vol(K)^k}
\]
is minimized iff $K$ is an ellipsoid.
\end{theorem}

Let $B_d$ be the $d$-dimensional unit ball, let $S_{d-1}$ be the boundary of $B_d$. Let $\kappa_d:= \vol(B_d) = \pi^{d/2}/\Gamma(1+\frac{d}{2})$, $\omega_d:= \vol_{d-1}(S_{d-1}) = d \kappa_d$.
\begin{theorem}[random simplex in ball, {\protect \cite[Theorem 8.2.3]{schneider}}]\label{thm:ball}
For any integer $k \geq 1$,
\[
\e \bigl( V_{B_d}^k \bigr) = \frac{1}{(d!)^k} \left(\frac{\kappa_{d+k}}{\kappa_d}\right)^{d+1} \frac{\kappa_{d(d + k + 1) }}{\kappa_{(d+1)(d + k)}} \frac{\omega_1 \dotsm \omega_k}{\omega_{d+1} \dotsm \omega_{d+k}}.
\]
\end{theorem}

\begin{theorem}[simplex with origin in ball, {\protect \cite[Theorem 8.2.2]{schneider}}]\label{thm:ballorigin}
For any integer $k \geq 1$,
\[
\e_{X_i \in B_d}\bigl( ( \vol \conv 0, X_1, \dotsc, X_d)^k \bigr) = \frac{1}{(d!)^k} \left(\frac{\kappa_{d+k}}{\kappa_d}\right)^{d} \frac{\omega_1 \dotsm \omega_k}{\omega_{d+1} \dotsm \omega_{d+k}}.
\]
\end{theorem}

\begin{lemma}[Busemann random simplex inequality,{\protect \cite[Theorem 8.6.1]{schneider}}]\label{lem:minimalWithOrigin}
Among all $d$-dimensional convex bodies,
\[
K \mapsto \frac{\e_{X_i \in K}\bigl( \vol \conv (0, X_1, \dotsc, X_i)\bigr)}{\vol(K)}
\]
is minimized iff $K$ is an ellipsoid centered at the origin. The minimum value is
\[
    \frac{1}{d!} \left(\frac{\kappa_{d+1}^d}{\kappa_d^{d+1}}\right) \frac{2}{ \omega_{d+1}}.
\]
\end{lemma}

\begin{lemma}[ball volume ratio, {\protect \cite[p. 455]{borgwardt}}]\label{lem:ratio}
\[
\sqrt{\frac{d}{2\pi}} \leq \frac{\kappa_{d-1}}{\kappa_{d}} \leq \sqrt{\frac{d+1}{2 \pi}}.
\]
\end{lemma}
\begin{proof}
The convexity of $\log \Gamma(x)$ implies
\[
(x+\alpha-1)^\alpha \leq \frac{\Gamma(x+\alpha)}{\Gamma(x)} \leq x^\alpha \qquad \text{for $\alpha \in (0,1)$ and $x \geq 1$}.
\]
The desired inequality follows.
\end{proof}

\section{Proofs}\label{sec:proofs}

\subsection{Proof of Theorem \ref{thm:detcov}}
We will first prove Theorem \ref{thm:detcov}. Most of the work is in proving the following dimension-dependent condition:
\begin{lemma}\label{lem:detCharacterization}
Monotonicity under inclusion of $K \mapsto \det A(K)$ holds for some dimension $d$ iff for any isotropic convex body $K \subseteq \RR^d $ we have $\sqrt{d} B_d \subseteq K$.
\end{lemma}
\begin{proof}
For the ``if'' part, suppose for a contradiction that $K, L \subseteq \RR^d$ are two convex bodies for which the monotonicity does not hold: $K \subseteq L$ but
\[
\det A(K) > \det A(L).
\]
By the continuity of $\det A(\cdot)$ and the density of polytopes we can assume without loss of generality that $K$ is a polytope satisfying the same properties. Let $m$ be the number of facets of $K$. Label the facets of $K$ arbitrarily with labels $1, \dotsc, m$. Let $F_i \subseteq \RR^d$, $i =0, \dotsc, m$ be the following non-increasing sequence of convex bodies: $F_0 = L$, $F_m = K$, $F_i = F_{i-1} \cap H_i$, where $H_i$ is the unique halfspace containing $K$ and containing facet $i$ of $K$ in its boundary. Then there exists $i$ such that $F_i, F_{i-1}$ is also a counterexample to the monotonicity. Let $v$ be the unit outer normal to facet $i$ of $K$. Consider the path from $F_{i-1}$ to $F_i$ induced by pushing $H_i$ in, formally, the path is given by, for $t \in [a,b]$:
\[
F(t) = F_{i-1} \cap H(t)
\]
where $H(t) = \{ x \in \RR^d \suchthat \inner{v}{x} \leq t\}$ and $a = \sup_{x \in F_{i-1}} \inner{v}{x}$, $b  = \sup_{x \in F_{i}} \inner{v}{x}$. The function $t \mapsto \det A(F(t))$ is continuous in $[a,b]$ and differentiable in $(a,b)$ and the lack of monotonicity implies that there exists $\bar t \in (a,b)$ such that its derivative is positive at $\bar t$. Now, even though this derivative is not invariant under affine transformations, its sign \emph{is} invariant. Thus, we can assume without loss of generality that $F_{\bar t}$ is in isotropic position and Proposition \ref{pro:derivativeAK} implies
\[
\e_{X \in S_{\bar t}} \bigl(\norms{X}\bigr) < d
\]
where $S_{\bar t} = F_{\bar t} \cap \bdry H(\bar t)$. In particular, there exists $x \in S_{\bar t}$ such that $\norm{x} < d$, which implies $\sqrt{d} B_d \nsubseteq F_{\bar t}$.

For the ``only if'' part, suppose that there is an isotropic convex body $K \subseteq \RR^d$ and a point $x \in \bdry K$ such that $\norm{x} < \sqrt{d}$. By continuity and an approximation argument, we can replace $K$ and $x$ without loss of generality so that $x$ is an extreme point of $K$, while still satisfying $\norm{x} < \sqrt{d}$ and isotropy.\footnote{For example, add a point $x_\alpha=\alpha x$ for $\alpha > 1 $ and take the convex hull between $K$ and $x_\alpha$ to get a convex body $K_\alpha$. We have that $x_\alpha$ is an extreme point of $K_\alpha$. For some $\alpha$ sufficiently close to $1$ and $T_\alpha = A(K_\alpha)^{-1/2}$, we have $T_\alpha K_\alpha$ isotropic and $\norm{ T_\alpha x_\alpha} < d$.} Let $v \in \RR^d \setminus \{0\}$ and $a < 0$ determine a halfspace $H = \{x \in \RR^n \suchthat \inner{v}{x} \geq a \}$ containing $K$ whose boundary intersects $K$ only at $x$. Let $H_t  = \{x \in \RR^n \suchthat \inner{v}{x} \geq t \}$. Let $L_t$ be the convex body $K \cap H_t$. Then by continuity, Proposition \ref{pro:derivativeAK} and the fact that $\norm{x} < d$, we have that there exists $\eps>0$ such that for all $t \in (a, a + \eps)$:
\[
\frac{d}{dt} \det A(L_t) > 0.
\]
This implies $\det A(K) < \det A(L_{a+\eps})$ while $L_{a+\eps} \subseteq K$.
\end{proof}

\begin{proof}[Proof of Theorem \ref{thm:detcov}.]
Immediate from Lemma \ref{lem:detCharacterization} and the fact that any $d$-dimensional isotropic convex body contains the ball of radius $\sqrt{(d+2)/d}$ centered at the origin and this is best possible \cite{MilmanPajor, sonnevend1989applications},\cite[Theorem 4.1]{KLS}.
\end{proof}

\subsection{Proof of Theorems \ref{thm:simplexMonotonicity} and \ref{thm:higherMoments}}

We will now prove Theorems \ref{thm:simplexMonotonicity} and \ref{thm:higherMoments}. We begin with a dimension-dependent condition similar to that of Lemma \ref{lem:detCharacterization}.
\begin{lemma}\label{lem:randomSimplexCharacterization}
For a given integer $k \geq 1$ and dimension $d$, monotonicity under inclusion of
\[
K \mapsto \e ( V_K^k )
\]
holds when $K$ ranges over $d$-dimensional convex bodies iff for any convex body $K \subseteq \RR^d $ and any $x \in \bdry K$ and $X_1, \dotsc, X_d$ random in $K$ we have
\begin{equation}\label{equ:randomSimplexCharacterization}
\e (V_K^k ) \leq \e \bigl( ( \vol \conv x, X_1, \dotsc, X_d )^k \bigr).
\end{equation}
\end{lemma}
\begin{proof}
The proof essentially the same as the proof of Lemma \ref{lem:detCharacterization}, with Proposition \ref{pro:derivativeAK} replaced by Proposition \ref{pro:Crofton}, with $q = d+1$,
\[
f(x_0, \dotsc, x_d) = (\vol \conv x_0, \dotsc, x_d)^k,
\]
and without using isotropy.
\end{proof}

Then we verify the dimension-dependent condition for $k=1$ in $\RR^2$ by means of the following lemma (which gives a lower bound to the right hand side of \eqref{equ:randomSimplexCharacterization}) and Blaschke's maximality of the triangle for Sylvester's problem, Theorem \ref{thm:blaschke} (which gives an upper bound to the left hand side).
\begin{lemma}\label{lem:plane}
Let $K \subseteq \RR^2$ be a convex body and let $x \in \bdry K$. Then
\begin{equation}\label{equ:plane}
\frac{\e_{X_1, X_2 \in K} (\vol \conv x, X_1, X_2)}{\vol K} \geq \frac{8}{9 \pi^2}.
\end{equation}
\end{lemma}
\begin{proof}
The continuity and affine-invariance of the lhs. of \eqref{equ:plane} and a standard compactness argument imply existence of a minimum $K$ and $x$.

To show the inequality, we will show with a series of symmetrizations that a half of a ball centered at $x$ minimizes the lhs. The intuition needed to understand the effect of Steiner symmetrization and Blaschke's shaking (see Section \ref{sec:preliminaries} for a brief review) is the following \cite[Section 3]{Pfiefer}:
If one picks three points at random from three vertical segments in the plane that are allowed to move vertically, picking one point from each segment, then the expected area of the convex hull of those 3 points is a strictly increasing function of the area of the triangle formed by the midpoints of the segments. For example, this implies that Steiner symmetrization decreases the expected area of a random triangle: The area of the triangle of the midpoints is zero when the midpoints lie on a common line.

Here is the sequence of symmetrizations:
\begin{enumerate}
\item Steiner symmetrization: Let $L$ be any supporting line of $K$ through $x$. Let $L^\perp$ be a line orthogonal to $L$ through $x$. Let $\bar K$ be the Steiner symmetrization of $K$ with respect to $L^\perp$. If $\bar K \neq K$, then $\bar K$ has a strictly smaller value than $K$ of the lhs. of \eqref{equ:plane}, see \cite[follows from Lemma 4]{Groemer} or \cite{Busemann}.

\item Blaschke's shaking (Sch\"{u}ttelung) with respect to $L$: for every chord of $\bar K$ perpendicular to $L$, shift it in the direction orthogonal to $L$ so that its endpoint that is nearest to $L$ lies on $L$. The union of the shifted chords is a convex body that we denote $\dbar K$. Lemma \ref{lem:shaking} shows that this operation can only decrease the value of the lhs. of \eqref{equ:plane}.
    So we know now that the set of pairs that are invariant under the previous step and this step contains a minimizer. Denote by $\mathcal{S}$ the family of pairs satisfying such invariance.

%
%
%

\item The lhs. of \eqref{equ:plane} just halves if one replaces $\dbar K$ with its symmetrization around $x$, and this symmetrization is a centrally symmetric convex body given the previous two steps. Thus,
\begin{equation}\label{equ:plane2}
\inf_{(K,x) \in \mathcal{S}} \frac{\e_K \vol \conv x, X_1, X_2}{\vol K} \geq 2 \inf_{K'} \frac{\e_{K'} \vol \conv 0, X_1, X_2}{\vol K'}
\end{equation}
where $K'$ ranges over all centrally symmetric convex bodies. Lemma \ref{lem:minimalWithOrigin} implies that ellipses are the only minimizers of the rhs. of \ref{equ:plane2}, and, as a half of an ellipse around the origin with the origin form a pair in $\mathcal{S}$, we conclude that half of a disk centered at $x$ is a minimizer.
\end{enumerate}
To get the rhs. in \eqref{equ:plane}, we just need to evaluate the lhs. for $x = 0$ and $K$ a half of the unit disk. For the numerator, the symmetry of the problem implies that the average for a half-disk and the origin is the same as the average for the disk and the origin. Thus, Theorem \ref{thm:ballorigin} implies
\[
\e \vol \conv x, X_1, X_2 = \frac{4}{9 \pi},
\]
while the denominator in \eqref{equ:plane} is the area of a half-disk, $\pi/2$.
\end{proof}
We believe that a half of an ellipse centered at $x$ is the only kind of minimizer.

\begin{proof}[Proof of Theorem \ref{thm:simplexMonotonicity}]
For the first part ($d \leq 2$), it is clearly true for $d=1$. For $d=2$, Theorem \ref{thm:blaschke} (Blaschke's maximality of the triangle for Sylvester's problem) and Lemmas \ref{lem:randomSimplexCharacterization} and \ref{lem:plane} imply the desired conclusion.

The second part ($d \geq 4$) is a special case of Theorem \ref{thm:higherMoments}.
%
\end{proof}
Numerical experiments suggest that a simplex and the center point of a facet work as a counterexample for the monotonicity as in Theorem \ref{thm:simplexMonotonicity} in $\RR^3$, and it should work in higher dimensions. Similar numerical experiments suggest that half of the unit ball and the origin is also a counterexample in $\RR^3$.

\begin{proof}[Proof of Theorem \ref{thm:higherMoments}]
Let $K$ be the half-ball with $x_d \geq 0$.

For $L$ the ball of volume $\vol(K)$, Theorem \ref{thm:blaschkegroemer} implies
\[
\e (V_K^k) \geq \e (V_L^k).
\]
Theorem \ref{thm:ball} implies
\begin{align*}
\e (V_L^k)
&=\frac{1}{2^k} \e(V_{B_d}^k\bigr) \\
&= \frac{1}{2^k (d!)^k} \left(\frac{\kappa_{d+k}}{\kappa_d}\right)^{d+1} \frac{\kappa_{d(d+k+1)}}{\kappa_{(d+1)(d+k)}} \frac{\omega_1 \dotsm \omega_k}{\omega_{d+1} \dotsm \omega_{d+k}}.
\end{align*}
On the other hand, symmetry and Theorem \ref{thm:ballorigin} imply
\begin{align*}
\e_{X_i \in K} \bigl((\vol \conv 0, X_1, \dotsc, X_d)^k \bigr) &= \e_{X_i \in B_d} \bigl((\vol \conv 0, X_1, \dotsc, X_d)^k \bigr) \\
    &= \frac{1}{(d!)^k} \left(\frac{\kappa_{d+k}}{\kappa_d}\right)^{d} \frac{\omega_1 \dotsm \omega_k}{\omega_{d+1} \dotsm \omega_{d+k}}.
\end{align*}
Combining the previous claims we get:
\begin{align}
\frac{\e_{X_i \in K} \bigl((\vol \conv 0, X_1, \dotsc, X_d)^k \bigr)}{\e (V_K^k)}
    &\leq 2^k \frac{\kappa_d}{\kappa_{d+k}} \frac{\kappa_{(d+1)(d+k)}}{\kappa_{d(d+k+1)}} \label{equ:moments}
\end{align}
When $d$ is 2 or 3, a tedious but straightforward use of Stirling's formula shows that \eqref{equ:moments} goes to 0 as $k$ goes to infinity. Lemma \ref{lem:randomSimplexCharacterization} completes the argument in this case.

If $d \geq 4$, Lemma \ref{lem:ratio} in \eqref{equ:moments} gives
\begin{multline*}
\frac{\e_{X_i \in K} \bigl((\vol \conv 0, X_1, \dotsc, X_d)^k \bigr)}{\e(V_K^k )} \\
\begin{aligned}
    &\leq 2^k \left(\frac{(d+2) \dotsm (d+k+1)}{\bigl(d(d+k+1)+1\bigr) \dotsm \bigl(d(d+k+1)+k\bigr)} \right)^{1/2} \\
    & \leq 2^k \left(\frac{d+k+1}{d(d+k+1)+k}\right)^{k/2},
\end{aligned}
\end{multline*}
(using the inequality $a/b \leq (a+1)/(b+1)$ whenever $0 \leq a \leq b$) and this is less than 1 for any $k \geq 1$. Lemma \ref{lem:randomSimplexCharacterization} completes the argument.
%
\end{proof}

\begin{lemma}\label{lem:shaking}
Let $K \subseteq \RR^2$ be a convex body, let $x \in \bdry K$, let $L$ be a supporting line of $K$ at $x$. Assume additionally that $K$ is symmetric around the line through $x$ orthogonal to $L$. Let $\bar K$ be Blaschke's shaking of $K$ with respect to $L$. Then
\begin{equation}\label{equ:shaking}
\e_{X_i \in K} (\vol \conv x, X_1, X_2) \geq \e_{X_i \in \bar K} (\vol \conv x, X_1, X_2).
\end{equation}
\end{lemma}
\begin{proof}
Without loss of generality, translate and rotate everything so that $x$ is at the origin and $L$ is the ``$x$'' axis. Let $t>0$ be half of the width of $K$ along the $x$ axis. For any $u \in [-t, t]$, define functions $\alpha(u)$ and $l(u)$ so that the vertical chords of $K$ have the form $\{u\} \times [\alpha(u), \alpha(u) + l(u)]$ (i.e., $\alpha$ is the ``bottom'' of the chord and $l$ is its length). We have
\begin{multline*}
\e_{X_i \in K} (\vol \conv 0, X_1, X_2) = \\
\frac{1}{2(\vol K)^2}\int_{-t}^t \int_{0}^{t} \int_0^{l_1(u_1)} \int_0^{l_2(u_2)} \abs{\det \rowmatrix{u1}{\alpha(u_1) + v_1}{u2}{\alpha(u_2)+v_2}} + \\
\abs{\det \rowmatrix{-u1}{\alpha(u_1) + v_1}{u2}{\alpha(u_2)+v_2}}
d v_2 d v_1 d u_2 d u_1.
\end{multline*}
Let $f(\alpha_1, \alpha_2)$ denote the integrand for fixed values of the integration variables:
\[
f(\alpha_1, \alpha_2) = \abs{\det \rowmatrix{u1}{\alpha_1 + v_1}{u2}{\alpha_2+v_2}} +
\abs{\det \rowmatrix{-u1}{\alpha_1 + v_1}{u2}{\alpha_2+v_2}}.
\]
The function $f$ is clearly convex. Moreover
\[
f(\alpha_1, \alpha_2) = f(-\alpha_1 - 2 v_1, \alpha_2) = f(\alpha_1, -\alpha_2 - 2 v_2) = f(-\alpha_1 - 2 v_1, -\alpha_2- 2 v_2).
\]
So, for $\lambda_i = \frac{\alpha_i}{2 (\alpha_i+ v_i)}$ and by convexity we have
\begin{align*}
f(\alpha_1, \alpha_2)
&= (1-\lambda_1) (1-\lambda_2) f(\alpha_1, \alpha_2) + (1-\lambda_1) \lambda_2 f(\alpha_1, -\alpha_2 - 2 v_2) \\
&\qquad + \lambda_1 (1-\lambda_2) f(-\alpha_1 - 2 v_1, \alpha_2) + \lambda_1 \lambda_2 f(-\alpha_1 - 2 v_1, -\alpha_2 - 2 v_2) \\
&\geq f(0,0).
\end{align*}
This in our integral gives
\begin{align*}
\e_{X_i \in K} (\vol \conv 0, X_1, X_2)
&\geq
\frac{1}{2(\vol K)^2}\int_{-t}^t \int_{0}^{t} \int_0^{l_1(u_1)} \int_0^{l_2(u_2)} \abs{\det \rowmatrix{u1}{v_1}{u2}{v_2}} +\\
&\qquad \abs{\det \rowmatrix{-u1}{v_1}{u2}{v_2}}
d v_2 d v_1 d u_2 d u_1\\
& = \e_{X_i \in \bar K} (\vol \conv 0, X_1, X_2).
\end{align*}
\end{proof}

\subsection{Crofton's formula and relatives}

\begin{proposition}[derivative of $\det A(K)$]\label{pro:derivativeAK}
Let $K \subseteq \RR^d$ be an isotropic convex body. Let $v \in \RR^d$ be a unit vector. Let $a = \inf_{x \in K} \inner{v}{x}$, $b = \sup_{x \in K} \inner{v}{x}$. Let $H_t = \{ x \in \RR^d \suchthat \inner{v}{x} \geq t\}$. Let $K_t = K \cap H_t$, $S_t = K \cap \bdry H_t$. Then
\[
\left.\frac{d}{dt} \det A(K_t)\right\rvert_{t=a} = \left( d - \e_{X \in S_a} \bigl(\norms{X}\bigr) \right) \frac{\vol_{d-1} S_a}{\vol K}.
\]
\end{proposition}
\begin{proof}
We have
\begin{align*}
A(K_t) &= \e_{X\in K_t} ((X- \mu(K_t))(X- \mu(K_t))^T) \\
&= \e_{X\in K_t} (X X^T) - \mu(K_t) \mu(K_t)^T.
\end{align*}
By isotropy, $\mu(K)=0$ and this implies
\begin{equation}\label{equ:dAk}
\begin{aligned}
\evaluatedat{\frac{d}{dt} A(K_t)}{t=a}
    &= \evaluatedat{\frac{d}{dt} \e_{X\in K_t} (X X^T)}{t=a}.
\end{aligned}
\end{equation}
Use the identity
\[
\frac{d}{dM} \det M = \left(M^{-1} \right)^T \det M
\]
to conclude
\begin{align*}
\frac{d}{dt} \det A(K_t)
    &= \evaluatedat{\frac{d}{dM} \det M}{M=A(K_t)} \cdot \frac{d}{dt} A(K_t) \\
    &= \det \bigl(A(K_t)\bigr) \left(A(K_t)^{-1}\right)^T \cdot \frac{d}{dt} A(K_t)
\end{align*}
where the dot ``$\cdot$'' represents the Frobenius inner product of matrices, $M \cdot N = \sum_{ij} M_{ij} N_{ij}$. This, isotropy and \eqref{equ:dAk} give
\begin{align*}
\evaluatedat{\frac{d}{dt} \det A(K_t)}{t=a}
    &= I \cdot \evaluatedat{\frac{d}{dt} \e_{X \in K_t} (X X^T)}{t=a} \\
    &= \evaluatedat{\frac{d}{dt} \e_{X \in K_t} (\norms{X})}{t=a}.
\end{align*}
To conclude, evaluate the following at $t=a$, using isotropy in the second step:
\begin{align*}
\frac{d}{dt} \e_{X \in K_t} (\norms{X})
    &= \frac{d}{dt} \frac{1}{\vol K_t} \int_t^b \e_{X \in S_\alpha} (\norms{X}) \vol_{d-1}(S_\alpha) \, d\alpha \\
    &= \frac{\vol_{d-1}(S_t)}{\vol K_t} \left( \e_{X \in K_t} (\norms{X}) - \e_{X \in S_t} (\norms{X}) \right).
\end{align*}
\end{proof}

We say that $f:U^q \to V$ is symmetric iff for any permutation $\pi$ of $\{1,\dotsc, q\}$ and any $x \in U^q$ we have $f(x) = f(x_{\pi(1)}, \dotsc, x_{\pi(q)})$.
\begin{proposition}[general derivative, Crofton]\label{pro:Crofton}
Let $K \subseteq \RR^d$ be a convex body. Let $v \in \RR^d$ be a unit vector. Let $a = \inf_{x \in K} \inner{v}{x}$, $b = \sup_{x \in K} \inner{v}{x}$. Let $H_t = \{ x \in \RR^d \suchthat \inner{v}{x} \geq t\}$. Let $K_t = K \cap H_t$, $S_t = K \cap \bdry H_t$. Let $f:{(\RR^d)}^q \to \RR$ be a symmetric continuous function. Let $X_1, \dotsc, X_q$ be independent random points in $K$. Then
\begin{multline*}
\left.\frac{d}{dt}\e f(X_1, \dotsc, X_q) \right\rvert_{t = a} \\
= q \Bigl( \e f(X_1, \dotsc, X_q) - \e \bigl(f(X_1, \dotsc, X_q) \giventhat X_1 \in S_a\bigr)\Bigr)\frac{\vol_{d-1} S_a}{\vol K}
\end{multline*}
\end{proposition}
(A slightly different proof should work with the weaker assumption that $f$ is bounded and measurable, but not necessarily continuous.)
\begin{proof}
Use repeatedly the identity
\[
\frac{d}{dt} \int_t^b u(x,t) dx = -u(t,t) + \int_t^b \frac{d}{dt} u(x,t) dx
\]
and the symmetry of $f$ to get
\begin{equation}\label{equ:Crofton}
\begin{aligned}
\left.\frac{d}{dt} \int_{K_t^q} f(x) dx \right\rvert_{t=a}
    &= \left.\frac{d}{dt} \int_{[t,b]^q} \int_{S_{\alpha_1} \times \dotsm \times S_{\alpha_q}} f(x) dx \, d\alpha \right\rvert_{t=a} \\
    &= - q \int_{[a,b]^{q-1}} \int_{S_a} \int_{S_{\alpha_2} \times \dotsm \times S_{\alpha_q}} f(x) dx_q \dotsm dx_2 \,dx_1 \, d\alpha_q \dotsm d\alpha_2
\end{aligned}
\end{equation}
Now,
\begin{align*}
&\left.\frac{d}{dt}\e_{X_i \in K_t} f(X_1, \dotsc, X_q) \right\rvert_{t = a} \\
    &= \left.\frac{d}{dt} \frac{1}{(\vol K_t)^q} \int_{K_t^q} f(x) dx \right\rvert_{t = a} \\
    &= \left.\frac{1}{(\vol K_t)^{2q}} \left( (\vol K_t)^q \left[\frac{d}{dt} \int_{K_t^q} f(x) dx \right] + q (\vol K_t)^{q-1} \vol_{d-1} (S_t) \int_{K_t^q} f(x) dx  \right) \right\rvert_{t = a} \\
    &= \frac{\vol_{d-1} (S_a)}{\vol K} \left( \frac{1}{(\vol K)^{q-1} \vol_{d-1} (S_a)} \left[\frac{d}{dt} \int_{K_t^q} f(x) dx \right]_{t = a} + \frac{q}{(\vol K)^{q} }  \int_{K^q} f(x) dx  \right)
\end{align*}
and to conclude use \eqref{equ:Crofton} and interpret the integrals as expectations.
\end{proof}

\subsection{Proof of Lemma \ref{lem:detVSsimplex}}
\begin{proof}[Proof of Lemma \ref{lem:detVSsimplex}]
If $Y$ is a random $d$-dimensional vector with second moments and $Y_1, \dotsc, Y_d$ are identically distributed independent copies of $Y$, then the following identity is known and easy to verify by expanding the determinant:
\begin{equation}\label{equ:blaschke}
\det \e Y Y^T = \frac{1}{d!} \e \left((\det Y_1, \dotsc, Y_d)^2\right).
\end{equation}
The first identity in the lemma follows immediately from this. To get the second identity (Equation \eqref{equ:detVSsimplex}), let $X_0$ be random in $K$ and consider:
\begin{align*}
V_K
&= \frac{1}{d!} \abs{\det (X_1-X_0, \dotsc, X_d-X_0)} \\
&=\frac{1}{d!} \abs{\det \begin{pmatrix} X_0 & \dotsm & X_d\\ 1 & \dotsm & 1
    \end{pmatrix}}.
\end{align*}
Taking expectation of the squares and using Equation \eqref{equ:blaschke} we get:
\begin{align*}
\e \left(V_K^2\right)
&= \frac{d+1}{d!} \det \begin{pmatrix}
    \e X X^T & \mu(K) \\
    \mu(K)^T & 1
\end{pmatrix}.
\end{align*}
The left hand side is invariant under translation of $K$, so the right hand side must be too and without loss of generality we can assume $\mu(K) = 0$. Equation \eqref{equ:detVSsimplex} follows.
\end{proof}

\section{Discussion}

A few open questions related to this work:
\begin{enumerate}
\item (Random polytopes) As mentioned in the introduction, Meckes and Reitzner asked for the monotonicity of the expected volume of a random polytope with $n$ vertices, not just a random simplex as in the current paper. It is easy to see that given $d$-dimensional convex bodies $K$, $L$ with $K \subset L$ there exists $n_0 = n_0(K,L)$ such that for $n \geq n_0$ we have
    \[
\e_{X_0, \dotsc, X_n \in K} \vol \conv X_0,\dotsc, X_n \leq \e_{X_0, \dotsc, X_n \in L} \vol \conv X_0,\dotsc, X_n.
    \]
    Can one choose $n_0$ so that it may depend on $d$ but is independent of $K$ and $L$?

%
\item (3-D case) For Meckes's strong conjecture, find an easy argument to disprove it for $d=3$.

\item (Slicing conjecture) Understand Meckes's weak conjecture.

\item (Sylvester's problem) Show that among all $d$-dimensional convex bodies,
\[
K \mapsto \frac{\e_{X_i \in K}( \vol \conv (X_0, \dotsc, X_d))}{\vol(K)}
\]
is maximized if $K$ is a simplex. (This is known to imply the slicing conjecture \cite{Giannopoulos}.)
\end{enumerate}

\section{Acknowledgments}
The author would like to thank Daniel Dadush, Navin Goyal, Mark Meckes and Santosh Vempala for suggesting some of these questions and helpful discussions.

\section*{Appendix}
For completeness, we prove here the equivalence between the slicing conjecture, Meckes's weak conjecture and Vempala's question. Slight variations of the following argument have been given by Mark Meckes \cite{Meckesp} and independently later by Santosh Vempala and Daniel Dadush \cite{DVp}. The main ingredients are Klartag's answer to the isomorphic slicing problem and a Khinchine-type inequality (reverse H\"older inequality).

It is known \cite{Ball1988}, \cite[Section 1.5]{Giannopoulos} that the slicing conjecture as stated in the introduction (in terms of hyperplane sections) is equivalent to the existence of a universal upper bound to the isotropic constant of a convex body, defined as follows: given a convex body $K \subseteq \RR^d$, the isotropic constant $L_K$ of $K$ is given by
\begin{equation}\label{equ:isotropicConstant}
    L_K^{2d} = \frac{\det A(K)}{(\vol K)^2}.
\end{equation}
\begin{conjecture}[slicing conjecture]\label{conj:slicing}
There exists a universal constant $c_3>0$ such that for any $d$ and any convex body $K \subseteq \RR^d$ we have $L_K \leq c_3$.
\end{conjecture}

We now state Klartag's result. For a pair of convex bodies $K, M \subseteq \RR^d$, let the Banach-Mazur distance be
\begin{align*}
d_{BM}(K,M) := \inf \{a &\geq 1 \suchthat K \subseteq T(M) \subseteq aK, \\
&\text{ $T :\RR^d \to \RR^d$ is a non-singular affine transformation}\}.
\end{align*}
\begin{theorem}[isomorphic slicing problem, \cite{Klartag}]\label{lem:klartag}
There exists $c>0$ such that if $K \subseteq \RR^d$ is a convex body and $\eps>0$, then there exists a convex body $M\subseteq \RR^d$ such that
\begin{itemize}
\item $d_{BM} (K,M) < 1+\eps$,
\item $L_M < \frac{c}{\sqrt{\eps}}$.
\end{itemize}
\end{theorem}

Here is the Khinchine-type inequality that wee need:
\begin{lemma}[{\cite[Appendix III]{MilmanSchechtman}} {\cite[Section 2.1]{Giannopoulos}} {\cite[p. 717]{GiannMilman}}]\label{lem:khinchine}
There exists a constant $c>0$ such that if $f : \RR^d \to \RR^+$ is a semi-norm, $K \subseteq \RR^d$ is a convex body and $1\leq p < \infty$, then
\[
\frac{1}{\vol K} \int_K f(x) \, dx \leq \left(\frac{1}{\vol K} \int_K f(x)^p \, dx \right)^{1/p} \leq \frac{c p}{\vol K}\int_K f(x) \, dx.
\]
\end{lemma}


The following proposition states the desired equivalences between the slicing conjecture and the monotonicity questions.
\begin{proposition}
For any $1 \leq p < \infty$, the following claims are equivalent:
\begin{enumerate}
\item (Vempala's question) There exists $c_1>0$ such that for any pair of convex bodies $K, M \subseteq \RR^d$ we have
\[
K \subseteq M \implies \det A(K) \leq c_1^{d} \det A(M).
\]
\item (Meckes's weak conjecture for $p$th moment) There exists $c_2>0$ such that for any pair of convex bodies $K, M \subseteq \RR^d$ we have
\[
K \subseteq M \implies \e (V_K^p) \leq c_2^d \e (V_M^p).
\]

\item (The slicing conjecture) Conjecture \ref{conj:slicing}.
\end{enumerate}
\end{proposition}
\begin{proof}
3 $\implies$ 1: Let $K, M \subseteq \RR^d$ be convex bodies such that $K \subseteq M$. Then (using Equations \eqref{equ:detVSsimplex} and \eqref{equ:isotropicConstant})
\begin{equation}\label{equ:constantVSsimplex}
\frac{\e V_K^2}{(\vol K)^2} = \frac{d+1}{d!} \frac{\det A(K)}{(\vol K)^2}
 = \frac{d+1}{d!} L_K^{2d},
\end{equation}
and we have a similar equality for $M$.

It is know that the isotropic constant has a universal lower bound $c >0$ over all dimensions and all convex bodies \cite{Bourgain1986,Ball1988,MilmanPajor}. This with our assumption implies $c \leq L_M, L_K \leq c_3$. That is (using Equation \eqref{equ:constantVSsimplex}),
\[
\frac{\e V_K^2}{(\vol K)^2} \leq \frac{d+1}{d!} c_3^{2d}
\]
and
\[
\frac{d+1}{d!} c^{2d} \leq \frac{\e V_M^2}{(\vol M)^2}.
\]
This and the fact that $\vol K \leq \vol M$ give
\[
\e \bigl(V_K^2\bigr) \leq \left(\frac{c_3}{c}\right)^{2d}\e \bigl( V_M^2 \bigr).
\]

1 $\implies$ 3:
%
%
%
The positive solution to the isomorphic slicing problem (Lemma \ref{lem:klartag}) with $\eps=1$ implies that there exists a constant $c>0$ such that for any convex body $K \subseteq \RR^d$ there exists another convex body $M \subseteq \RR^d$ satisfying $d_{BM}(K,M) \leq 2$ and $L_M \leq c$. Consider an arbitrary convex body $K \subseteq \RR^d$ and let $M$ be the convex body given by the lemma. As the isotropic constants $L_K$, $L_M$ and $\det A(\cdot)$ are invariant under affine transformations, we can assume without loss of generality that $K \subseteq M \subseteq 2 K$. This and \eqref{equ:isotropicConstant} imply
\[
L_K^{2d} = \frac{\det A(K)}{(\vol K)^2} \leq 2^{2d} \frac{ c_1^d \det A(M)}{(\vol M)^2}
\leq (2 c \sqrt{c_1})^{2d}.
\]

1 $\Leftrightarrow$ 2: This is an easy consequence of our Khinchine-type inequality (Lemma \ref{lem:khinchine}) and Equation \eqref{equ:constantVSsimplex}: Iterated use of Lemma \ref{lem:khinchine} implies
\[
\e V_K \leq \bigl(\e( V_K^p ) \bigr)^{1/p} \leq c^{d+1} p^{d+1} \e V_K.
\]
The claimed equivalence follows.
\end{proof}

\bibliographystyle{abbrv}    
\bibliography{sylvester}

\end{document}